\documentclass[english]{jams-l}

\newtheorem{theorem}{Theorem}[section]
\newtheorem{lemma}[theorem]{Lemma}

\newtheorem{proposition}[theorem]{Proposition}
\newtheorem{corollary}[theorem]{Corollary}

\theoremstyle{definition}
\newtheorem{definition}[theorem]{Definition}

\numberwithin{equation}{section}

\theoremstyle{remark}
\newtheorem*{remark}{Remark}

\usepackage{graphicx}

\usepackage[fixlanguage]{babelbib}
\selectbiblanguage{english}

\usepackage[shortlabels]{enumitem}

\usepackage{amsfonts,amssymb,amsbsy,amsmath,mathrsfs,amsthm}

\usepackage{commath}

\usepackage{hyperref}
\hypersetup{pdfpagemode=UseNone, pdfstartview=FitH,
  colorlinks=true,urlcolor=red,linkcolor=blue,citecolor=black,
  pdftitle={Default Title, Modify},pdfauthor={Your Name},
  pdfkeywords={Modify keywords}}

\RequirePackage{babel}

\providecommand{\norm}[1]{ \lVert#1  \rVert}
\newcommand{\dmu}{\, d \mu}

\newcommand{\dla}{\, d \lambda}
\newcommand{\lp}{\left (}
\newcommand{\rp}{\right )}

\def\Xint#1{\mathchoice
   {\XXint\displaystyle\textstyle{#1}}%
   {\XXint\textstyle\scriptstyle{#1}}%
   {\XXint\scriptstyle\scriptscriptstyle{#1}}%
   {\XXint\scriptscriptstyle\scriptscriptstyle{#1}}%
   \!\int}
\def\XXint#1#2#3{{\setbox0=\hbox{$#1{#2#3}{\int}$}
     \vcenter{\hbox{$#2#3$}}\kern-.5\wd0}}

\def\dashint{\Xint-}

\usepackage{cite}

\begin{document}

\title{Median-type John--Nirenberg space in metric measure spaces}

\author{Kim Myyryl\"{a}inen}
\address{Department of Mathematics, Aalto University, P.O. Box 11100, FI-00076 Aalto, Finland}
\email{kim.myyrylainen@aalto.fi}
\thanks{The author would like to thank Juha Kinnunen and Riikka Korte for valuable discussions. The author would also like to thank the anonymous referee for carefully reading the paper and for constructive comments. The research was supported by the Academy of Finland.}



\begin{abstract}
We study the so-called John--Nirenberg space that is a generalization of functions of bounded mean oscillation in the setting of metric measure spaces with a doubling measure. Our main results are local and global John--Nirenberg inequalities, which give weak type estimates for the oscillation of a function. We consider medians instead of integral averages throughout, and thus functions are not a priori assumed to be locally integrable. Our arguments are based on a Calder\'{o}n-Zygmund decomposition and a good-$\lambda$ inequality for medians. A John--Nirenberg inequality up to the boundary is proven by using chaining arguments. As a consequence, the integral-type and the median-type John--Nirenberg spaces coincide under a Boman-type chaining assumption.

\end{abstract}

\maketitle

\section{Introduction}

The space of functions of bounded mean oscillation (BMO) was introduced by John and Nirenberg in~\cite{john_original}. 
In that article, they also discussed a larger BMO-type space called the John--Nirenberg space, denoted by $JN_p$ with $1<p<\infty$. 
The John--Nirenberg space contains BMO, and BMO is obtained as the limit of $JN_p$ as $p \to \infty$. 
The John--Nirenberg lemma in~\cite{john_original} states that a logarithmic blowup is the worst possible for a BMO function.
Moreover, another version of the John-Nirenberg lemma in~\cite{john_original} implies that $JN_{p}$ is a subset of weak $L^{p}$.
This inclusion is strict, as shown by a one-dimensional example in~\cite{aalto}.
It is also known that $L^p$ is a proper subset of $JN_p$; see~\cite{korte}.

In~\cite{johnmedian}, John discussed BMO using medians instead of integral averages.
He also proved the analogous John--Nirenberg inequality for the median-type BMO in a Euclidean space. Later, Str\"{o}mberg proved the inequality for a larger class of medians and also gave a proof in a metric measure space with a doubling measure~\cite{stromberg, weightedhardy}. 
In particular, this implies that the ordinary BMO and the median-type BMO are equivalent.
While BMO has been studied extensively, the John--Nirenberg space is not equally well understood.
However, some results related to John--Nirenberg inequalities and interpolation of operators can be found in~\cite{campanato1966,stampacchia1965, giaquinta, giaquinta_martinazzi, giusti}.

Our main results are John--Nirenberg inequalities for the median-type John--Nirenberg space.
We carry out the analysis in a metric measure space with a doubling measure, but the main results are new even in the Euclidean setting.
The main novelty is that we consider medians instead of integral averages throughout.
One of the advantages of using medians is that we do not need to assume that the function is locally integrable.
Thus, the definition via medians applies to any measurable function. 
There exist different definitions of John--Nirenberg spaces in metric measure spaces; see~\cite{aalto, kinnunen, marolasaari}. 
The difference is whether the balls in the definition are required to be pairwise disjoint or allowed to overlap.
Definitions with disjoint balls can be found in~\cite{kinnunen, marolasaari} and with overlap in~\cite{aalto}.
We adopt the definition in~\cite{kinnunen, marolasaari} which leads to a more general theory.

The proof of the local John--Nirenberg inequality with medians (Theorem~\ref{thm1}) in Section 4 is based on a Calder\'{o}n-Zygmund decomposition (Lemma~\ref{lemma3}) and a good-$\lambda$ inequality (Lemma~\ref{lemma4}) for medians.
This is inspired by~\cite{aalto, kinnunen}, where the authors examine integral-type John--Nirenberg spaces in metric measure spaces with a doubling measure.
A challenge in proving results for the median-type John--Nirenberg space is that medians lack subadditive properties and monotonicity on sets compared to integral averages.
Our Theorem~\ref{thm1} implies a related John--Nirenberg inequality in~\cite[Theorem 1.3]{lerner_perez} where a discrete summability condition is considered; see also~\cite{kinnunen, franchi_perez_wheeden, macmanus_perez}.
A simple one-dimensional example shows that our result is more general than~\cite[Theorem 1.3]{lerner_perez}.
In particular, our approach does not depend on the discrete summability condition and rearrangements.

In Section 5, we prove a John--Nirenberg inequality with medians up to the boundary (Theorem~\ref{JN-lemma_Boman}) in Boman sets by applying the local John--Nirenberg inequality together with chaining arguments.
Boman sets are closely related to the Boman chain condition, introduced in the unpublished paper~\cite{bomanlp}.
The connection is discussed in~\cite{marolasaari}.
In particular, the two are equivalent in a geodesic space with a doubling measure.
The Boman chain condition characterizes John domains in many metric measure spaces, including Euclidean spaces; see~\cite{boman_eq_john}.
The corresponding results for the integral-type John--Nirenberg space can be found in~\cite{johndomain, marolasaari}.
We apply the arguments in~\cite{chua, johndomain, marolasaari} for medians. 
As a corollary of the global John--Nirenberg inequality, we show that the integral- and median-type John--Nirenberg spaces coincide in every open set 
under the assumption that balls are Boman sets with uniform parameters (Corollary \ref{equivalence_Boman}).
This means that the median-type John--Nirenberg condition is possibly the weakest for a function to be in $JN_p$.
The uniform Boman condition on balls holds, for example, in geodesic spaces~\cite{koskela}.

\section{Preliminaries}

Let $(X,d,\mu)$ be a metric measure space with a metric $d$ and a doubling measure~$\mu$. 
A Borel regular measure is said to be doubling if
\[
0 < \mu(2B) \leq c_\mu \mu(B) < \infty
\]
for every ball $B = B(x,r) = \{ y\in X: d(x,y) < r \}$, where $c_\mu>1$ is the doubling constant.
We use the notation $\lambda B = B(x,\lambda r)$, $\lambda>0$, for the $\lambda$-dilate of $B$.
From the doubling property of the measure, it can be deduced that if $y \in B(x,R) \subset X$ and $0<r\leq R <\infty$, then
\[
\frac{\mu(B(x,R))}{\mu(B(y,r))} \leq c_\mu^2 \lp \frac{R}{r} \rp^D,
\]
where $D = \log_2 c_\mu$ is the doubling dimension of the space $(X,d,\mu)$. The proof can be found in~\cite[p.~6]{bjorn}.
We denote the integral average of a function $f \in L^1(A)$ in a set $A \subset X$ by
\[
f_A = \dashint_A f \dmu = \frac{1}{\mu(A)} \int_A f \dmu .
\]

A basic tool in metric measure spaces is the 5-covering theorem. It is also sometimes referred to as the basic covering theorem or Vitali's covering theorem. Although, one must be careful since there is another covering theorem named after Vitali~\cite[pp.~3--4]{heinonen}. One can check~\cite[pp.~2--3]{heinonen} for a proof.

\begin{lemma}
Let $\mathcal{F}$ be a collection of balls of uniformly bounded radii in $X$. Then there exists a countable disjointed subcollection $\mathcal{G}$ such that
\[
\bigcup_{B \in \mathcal{F}} B \subset \bigcup_{B \in \mathcal{G}} 5B .
\]
\end{lemma}

The Lebesgue differentiation theorem states that the integral average of a function $f$ over a ball $B(x,r)$ approaches $f(x)$ when the radius $r$ tends to zero. Two different proofs can be found in~\cite[pp.~4--6,~12--13]{heinonen}.

\begin{lemma}
\label{leb.diff.thm}
If $f \in L_{\text{\normalfont loc}}^1(X)$, then it holds that
\[
\lim_{r \to 0} \dashint_{B(x,r)} |f-f(x)| \dmu = 0
\]
for $\mu$-almost every $x \in X$.
\end{lemma}

We follow the definition in~\cite{kinnunen, marolasaari} for the integral-type John--Nirenberg space $JN_{p,q}$ in metric measure spaces.

\begin{definition}
\label{JNp}
Let $\Omega \subset X$ be an open set, $1<p<\infty$ and $0<q<p$. We say that a function $f \in L_{\text{loc}}^q(\Omega)$ belongs to the John-Nirenberg space $JN_{p,q}(\Omega)$ if
\[
\norm{f}_{JN_{p,q}(\Omega)}^p = \sup \sum_{i=1}^{\infty} \mu(B_i) \lp \inf_{c_i \in \mathbb{R}} \dashint_{B_i} |f-c_i|^q \dmu \rp^\frac{p}{q}  < \infty  ,
\]
where the supremum is taken over countable collections of pairwise disjoint balls $B_i \subset \Omega$.

\end{definition}

If $q=1$, we write $JN_p$ instead of $JN_{p,1}$.
We omit $\Omega$ from the norms if the considered set is clear from the context.

Next, we present medians and discuss their properties.
Medians have been studied and used in different problems of analysis; see for example~\cite{federer_herbert_william, fujii, gogatishvili, heikkinen, heikkinen_measuredensity, heikkinen_kinnunen, approxquasi, approxholder, jawerth_perez_welland, jawerth_torchinsky, johnmedian, karak, lerner, lerner_perez, medians, stromberg, weightedhardy, zhou}.

\begin{definition}
Let $A \subset X$ be a set of finite and positive measure, $0<s\leq 1$ and $f:X \rightarrow [-\infty, \infty]$ be a measurable function.
If a value $M_f^s(A)$ satisfies
\[
\mu(\{x \in A: f(x) > M_f^s(A)\}) \leq s \mu(A)
\]
and
\[
\mu(\{x \in A: f(x) < M_f^s(A)\}) \leq (1-s) \mu(A),
\]
then we call $M_f^s(A)$ an $s$-median of the function $f$ over a set $A$.

\end{definition}

Note that a $\tfrac{1}{2}$-median is a standard median value of $f$ over $A$.
An $s$-median of a function is not always unique. For example, consider $f = \chi_{[1/2, 1]}$ on the interval $[0,1]$. Then any value between 0 and 1 is a $\tfrac{1}{2}$-median of $f$. Thus, we define the maximal $s$-median $m_f^s(A)$ which is unique~\cite{medians}.

\begin{definition}
\label{maxmedian}
Let $0<s \leq 1$ and $A \subset X$ be such that $0<\mu(A)<\infty$.
The maximal $s$-median of a measurable function $f:X \rightarrow [-\infty, \infty]$ over a set $A$ is defined as
\[
m_f^s(A) = \inf \{a \in \mathbb{R} : \mu(\{x \in A: f(x) > a\}) < s \mu(A) \} .
\]
\end{definition}

It can be shown that the maximal $s$-median of a function is indeed an $s$-median~\cite{medians}.
In the next lemma, we list the basic properties of the maximal $s$-median.
Most of these properties are listed without proofs in~\cite{approxquasi, approxholder}.
The proofs of properties (i), (ii), (v), (vii), (viii) and (ix) can be found in~\cite[Proposition~1.1]{medians} in the Euclidean setting. The proofs of these properties are practically same in metric measure spaces, and thus are omitted here.
We give proofs for the remaining properties.

\begin{lemma}
\label{medianprops}
Let $A \subset X$ such that $0 < \mu(A) < \infty$,  $f,g:X \rightarrow [-\infty, \infty]$ be measurable functions and $0<s\leq 1$.
Then the maximal $s$-median has the following properties:
\begin{enumerate}[(i),topsep=5pt,itemsep=5pt]

\item If $s \leq s'$, then $m_f^{s'}(A) \leq m_f^s(A)$.

\item $m_f^s(A) \leq m_g^s(A)$ whenever $f\leq g$ $\mu$-almost everywhere in $A$.

\item If $A \subset A'$ and $\mu(A') \leq c\mu(A)$ with some $c \geq 1$, then $m_f^s(A) \leq m_f^{s/c}(A')$.

\item $m_{\varphi \circ f}^s(A) = \varphi(m_f^s(A))$ for an increasing continuous function $\varphi: f[A] \to [-\infty, \infty]$.

\item $m_f^s(A) + c = m_{f+c}^s(A)$ for $c \in \mathbb{R}$.

\item $m_{cf}^s(A) = c \, m_f^s(A)$ for $c > 0$.

\item $|m_{f}^s(A)| \leq m_{|f|}^{\min\{s,1-s\}}(A)$.

\item $m_{f+g}^s(A) \leq m_f^{t_1}(A) + m_g^{t_2}(A)$ whenever $t_1 + t_2 \leq s$.

\item For $f \in L^p(A)$ and $p>0$, \[ m_{|f|}^s(A) \leq \lp s^{-1} \dashint_A |f|^p \dmu \rp^{\frac{1}{p}}. \]

\item If $A_i$ are pairwise disjoint for every $i \in \mathbb{N}$, then
\[
\inf_{i} m_f^s(A_i) \leq m_f^s\big(\bigcup_{i=1}^\infty A_i\big) \leq \sup_{i} m_f^s(A_i) .
\]

\end{enumerate}
\end{lemma}

\begin{proof}

(iii) If $a > m_f^{s/c}(A')$, then we have 
\[
\mu(\{x \in A: f(x) > a\}) \leq \mu(\{x \in A': f(x) > a\}) < \frac{s}{c} \mu(A') \leq s \mu(A) ,
\]
from which the claim follows.

(iv) It holds that
\begin{align*}
\varphi(m_f^s(A)) &= \inf \{\varphi(a) \in \mathbb{R} : \mu(\{x \in A: f(x) > a\}) < s \mu(A) \} \\
&=  \inf \{\varphi(a) \in \mathbb{R} : \mu(\{x \in A: \varphi(f(x)) > \varphi(a)\}) < s \mu(A) \} \\ &= m_{\varphi \circ f}^s(A) .
\end{align*}

(vi) Apply property (iv) for $\varphi(x) = cx$.

(x) For $a > \sup_i m_f^s(A_i)$, we have 
\begin{align*}
\mu\big(\big\{x \in \bigcup_{i=1}^\infty A_i: f(x) > a \big\}\big) &= \sum_{i=1}^\infty \mu(\{x \in A_i: f(x) > a\}) \\
&< \sum_{i=1}^\infty  s \mu(A_i) = s \mu\big(\bigcup_{i=1}^\infty A_i\big).
\end{align*}
It follows that $m_f^{s}\big(\bigcup_{i=1}^\infty A_i\big) \leq a$. Since this holds for every $a > \sup_i m_f^s(A_i)$, we have 
\[
m_f^{s}\big(\bigcup_{i=1}^\infty A_i\big) \leq \sup_i m_f^s(A_i).
\]
To prove the other inequality, assume that $a< \inf_i m_f^s(A_i)$. We then get
\begin{align*}
\mu\big(\big\{x \in \bigcup_{i=1}^\infty A_i: f(x) < a \big\}\big) &= \sum_{i=1}^\infty \mu(\{x \in A_i: f(x) < a\}) \\
&< \sum_{i=1}^\infty  (1-s) \mu(A_i) = (1-s) \mu\big(\bigcup_{i=1}^\infty A_i\big).
\end{align*}
This implies that $a \leq m_f^{s}\big(\bigcup_{i=1}^\infty A_i\big)$ for every $a< \inf_i m_f^s(A_i)$. Hence, we obtain
\[
\inf_i m_f^s(A_i) \leq m_f^{s}\big(\bigcup_{i=1}^\infty A_i\big).
\]

\end{proof}

\begin{remark}
Assume that $0<s\leq 1/2$. Then property (vii) assumes a slightly simpler form
\[
|m_{f}^s(A)| \leq m_{|f|}^{\min\{s,1-s\}}(A) = m_{|f|}^s(A) ,
\]
since 
\[
m_{|f|}^{1-s}(A) \leq m_{|f|}^s(A)
\]
for $0<s\leq 1/2$.
\end{remark}

Suppose that $f \in L_{\text{loc}}^1(X)$ and $0<s\leq1/2$. 
Using the Lebesgue differentiation theorem (Lemma~\ref{leb.diff.thm}) together with properties (v), (vii) and (ix) of the maximal $s$-medians, we obtain the following version of the Lebesgue differentiation theorem:
\begin{align*}
|m_f^s(B(x,r)) - f(x)| &= |m_{f-f(x)}^s(B(x,r))| \\
&\leq m_{|f-f(x)|}^s(B(x,r)) \\ &\leq s^{-1} \dashint_{B(x,r)} |f-f(x)| \dmu \to 0
\end{align*}
as $r \to 0$.
However, there is a more general version of the Lebesgue differentiation theorem for medians~\cite{medians} where we need to assume only that $f$ is a measurable function.

The proof of Lemma~\ref{leb.diff.medians} can be found in~\cite[Theorem~2.1]{medians} where the lemma is proven in the Euclidean setting. The proof is almost identical in metric measure spaces with a doubling measure, and thus is omitted here.

\begin{lemma}
\label{leb.diff.medians}
Let $f:X \rightarrow [-\infty, \infty]$ be a measurable function which is finite almost everywhere in $X$ and $0<s\leq 1$. Then it holds that
\[
\lim_{r \to 0} m_f^s(B(x,r)) = f(x)
\]
for $\mu$-almost every $x \in X$.

\end{lemma}

We give a definition for the median-type BMO which coincides with the classical BMO~\cite{johnmedian, stromberg, weightedhardy}.

\begin{definition}
We say that a measurable function $f$ belongs to BMO$_{0,s}(X)$ if
\[
\norm{f}_{\text{BMO}_{0,s}(X)} = \sup_{B \subset X} \inf_{c} m_{|f-c|}^s(B) < \infty .
\]
\end{definition}

\section{Definition and properties of $JN_{p,0,s}$}

In this section, we give a definition of the median-type John--Nirenberg space in metric measure spaces.
Moreover, we examine the basic properties of the space.

\begin{definition}
\label{Def_medJN}
Let $\Omega \subset X$ be an open set, $1<p<\infty$ and $0<s \leq 1/2$. We say that a measurable function $f$ belongs to the median-type John--Nirenberg space $JN_{p,0,s}(\Omega)$ if
\[
\norm{f}_{JN_{p,0,s}(\Omega)}^p = \sup \sum_{i=1}^{\infty} \mu(B_i) \lp \inf_{c_i \in \mathbb{R}} m_{|f-c_i|}^s (B_i) \rp^p < \infty ,
\]
where the supremum is taken over all countable collections of pairwise disjoint balls $B_i \subset \Omega$.
\end{definition}

The zero in the notation $JN_{p,0,s}$ means that we do not need to assume any local integrability but only measurability.
The range $0<s \leq 1/2$ is necessary since $\norm{f}_{JN_{p,0,s}} = 0$ for $s>1/2$ and a two-valued function $f$.

The next lemma shows that the constants $c_i$ in the definition of $JN_{p,0,s}$ can be replaced by the maximal $t$-medians where $s\leq t \leq \tfrac{1}{2}$.

\begin{lemma}
\label{lemmamedian1}
Let $f$ be a measurable function. It holds that
\[
\norm{f}_{JN_{p,0,s}(\Omega)}^p \leq \sup \sum_{i=1}^{\infty} \mu(B_i) \lp m_{|f-m_f^t(B_i)|}^s (B_i) \rp^p \leq 2^p \norm{f}_{JN_{p,0,s}(\Omega)}^p ,
\]
whenever $0<s\leq t \leq 1/2$.

\end{lemma}

\begin{proof}
It is clear that the first inequality holds. The other inequality follows from
\begin{align*}
m_{|f-m_{f}^t(B_i)|}^s(B_i) &\leq m_{|f-c_i| + |m_{f}^t(B_i)-c_i|}^s(B_i) \\ &= m_{|f-c_i|}^s(B_i) + |m_{f}^t(B_i)-c_i| \\
&= m_{|f-c_i|}^s(B_i) + |m_{f-c_i}^t(B_i)| \\
&\leq m_{|f-c_i|}^s(B_i) + m_{|f-c_i|}^t(B_i) \\
&\leq 2 m_{|f-c_i|}^s(B_i).
\end{align*}
By taking the infimum over the constants $c_i$, we observe that
\[
m_{|f-m_{f}^t(B_i)|}^s(B_i) \leq 2 \inf_{c_i} m_{|f-c_i|}^s(B_i) .
\]
Therefore, the second inequality is attained as well.

\end{proof}

Next, we list some basic properties of $JN_{p,0,s}$ spaces.

\begin{lemma}
\label{lemma5}
Let $f$ and $g$ be measurable functions. Then the following properties hold true:
\begin{enumerate}[(i), parsep=5pt, topsep=5pt]
    \item $\norm{f+g}_{JN_{p,0,s}} \leq \ \norm{f}_{JN_{p,0,t_1}} + \norm{g}_{JN_{p,0,t_2}}$ whenever $t_1 + t_2 \leq s$.
    \item $\norm{|f|}_{JN_{p,0,s}} \leq \norm{f}_{JN_{p,0,s}}$.
    \item If $2t_1 + 2t_2 \leq s$, then
    \begin{align*}
    \norm{\max\{f,g\}}_{JN_{p,0,s}} &\leq \norm{f}_{JN_{p,0,t_1}} + \norm{g}_{JN_{p,0,t_2}} , \\
    \norm{\min\{f,g\}}_{JN_{p,0,s}} &\leq \norm{f}_{JN_{p,0,t_1}} + \norm{g}_{JN_{p,0,t_2}} .
    \end{align*}
     
\end{enumerate}

\end{lemma}

\begin{proof}

(i) Using property (viii) of Lemma~\ref{medianprops} and Minkowski's inequality, we estimate
\begin{align*}
& \left[ \sum_{i=1}^{\infty} \mu(B_i) \lp \inf_{c_i} m_{|f+g-c_i|}^s (B_i) \rp^p \right]^\frac{1}{p} \\ 
&\leq \left[ \sum_{i=1}^{\infty} \mu(B_i) \lp  m_{|f+g - c_i^f - c_i^g|}^s (B_i) \rp^p \right]^\frac{1}{p} \\ 
&\leq \left[ \sum_{i=1}^{\infty} \mu(B_i) \lp  m_{|f- c_i^f|}^{t_1} (B_i) + m_{|g - c_i^g|}^{t_2} (B_i) \rp^p \right]^\frac{1}{p} \\ 
&= \left[ \int_{\Omega} \lp  \sum_{i=1}^{\infty} \chi_{ B_i}(x) \, m_{|f- c_i^f|}^{t_1} (B_i) + \sum_{i=1}^{\infty} \chi_{ B_i}(x) \, m_{|g - c_i^g|}^{t_2} (B_i) \rp^p \dmu(x) \right]^\frac{1}{p} \\ 
&\leq \left[ \int_{\Omega} \lp  \sum_{i=1}^{\infty} \chi_{ B_i}(x) \, m_{|f- c_i^f|}^{t_1} (B_i) \rp^p \dmu(x) \right]^\frac{1}{p} + \left[ \int_{\Omega} \lp \sum_{i=1}^{\infty} \chi_{B_i}(x) \, m_{|g - c_i^g|}^{t_2} (B_i) \rp^p \dmu(x) \right]^\frac{1}{p} \\ 
&= \left[ \sum_{i=1}^{\infty} \mu(B_i) \lp m_{|f- c_i^f|}^{t_1} (B_i) \rp^p \right]^\frac{1}{p} + \left[ \sum_{i=1}^{\infty} \mu(B_i) \lp m_{|g - c_i^g|}^{t_2} (B_i) \rp^p \right]^\frac{1}{p}
\end{align*}
for any $c_i^f, c_i^g \in \mathbb{R}$.
Thus, we can take the infimum over the constants $c_i^f$ and $c_i^g$ to get
\begin{align*}
\norm{f+g}_{JN_{p,0,s}} \leq \ \norm{f}_{JN_{p,0,t_1}} + \norm{g}_{JN_{p,0,t_2}} ,
\end{align*}
where $t_1 + t_2 \leq s$.

(ii) The triangle inequality together with Lemma~\ref{medianprops}~(ii) implies
\begin{align*}
\sum_{i=1}^{\infty} \mu(B_i) \lp \inf_{c_i} m_{||f|-c_i|}^s (B_i) \rp^p &\leq \sum_{i=1}^{\infty} \mu(B_i) \lp \inf_{c_i} m_{||f|-|c_i||}^s (B_i) \rp^p \\
&\leq \sum_{i=1}^{\infty} \mu(B_i) \lp \inf_{c_i} m_{|f-c_i|}^s (B_i) \rp^p .
\end{align*}
Therefore, we obtain the claim
\[
\norm{|f|}_{JN_{p,0,s}} \leq \norm{f}_{JN_{p,0,s}} .
\]

(iii) Note that $\max\{f,g\} = \tfrac{1}{2}(f+g+|f-g|)$ and $\min\{f,g\} = \tfrac{1}{2}(f+g-|f-g|)$.
Then using (i) and (ii), we get
\begin{align*}
\norm{\max\{f,g\}}_{JN_{p,0,s}} &\leq \frac{1}{2} \lp \norm{f+g}_{JN_{p,0,t_1 + t_2}} + \norm{|f-g|}_{JN_{p,0,t_1 + t_2}} \rp \\
&\leq \frac{1}{2} \lp \norm{f}_{JN_{p,0,t_1}} + \norm{g}_{JN_{p,0,t_2}} + \norm{f-g}_{JN_{p,0,t_1 + t_2}} \rp \\
&\leq \frac{1}{2} \lp \norm{f}_{JN_{p,0,t_1}} + \norm{g}_{JN_{p,0,t_2}} + \norm{f}_{JN_{p,0,t_1}} + \norm{g}_{JN_{p,0,t_2}} \rp \\
&= \norm{f}_{JN_{p,0,t_1}} + \norm{g}_{JN_{p,0,t_2}} ,
\end{align*}
where $2t_1 + 2t_2 \leq s$.
The claim for $\min\{f,g\}$ follows similarly.\qedhere

\end{proof}

The next proposition tells that the space $L^p$ is contained in $JN_{p,q}$ which in turn is a subset of $JN_{p,0,s}$. 
The first inclusion is strict in the Euclidean setting, that is, there exists a function in $JN_p \setminus L^p$~\cite{korte}.
The second one holds in the other direction in many situations; see Corollary~\ref{equivalence_Boman}.

\begin{proposition}
\label{LpinJNp}
Let $1 < p< \infty$, $0 < q < p$ and $0 < s \leq 1/2$.
It holds that $L^p(\Omega) \subset JN_{p,q}(\Omega) \subset JN_{p,0,s}(\Omega)$, particularly
\[
s^\frac{1}{q} \norm{f}_{JN_{p,0,s}(\Omega)} \leq \norm{f}_{JN_{p,q}(\Omega)} \leq \norm{f}_{L^p(\Omega)} .
\]

\end{proposition}

\begin{proof}

The first inequality follows straightforwardly from property (ix) of Lemma~\ref{medianprops}.
The second one is obtained by a simple use of H\"{o}lder's inequality:
\begin{align*}
\sum_{i=1}^\infty \mu(B_i) \lp \inf_{c_i} \dashint_{B_i} |f-c_i| \dmu \rp^p
&\leq \sum_{i=1}^\infty \mu(B_i) \lp \dashint_{B_i} |f| \dmu \rp^p \\
&\leq \sum_{i=1}^\infty \mu(B_i) \dashint_{B_i} |f|^p \dmu \\
&\leq \int_{\Omega} |f|^p \dmu .
\end{align*}
By taking the supremum over all collections of pairwise disjoint balls $B_i \subset \Omega$, we get
\[
\norm{f}_{JN_{p,q}(\Omega)} \leq \norm{f}_{L^p(\Omega)} .
\]

\end{proof}

\begin{remark}
If $f \in \text{BMO}_{0,s}(\Omega)$ with $\mu(\Omega) < \infty$, then it holds that $f \in JN_{p,0,s}(\Omega)$.
More precisely, we have
\[
\norm{f}_{JN_{p,0,s}(\Omega)} \leq \mu(\Omega)^\frac{1}{p} \norm{f}_{\text{BMO}_{0,s}(\Omega)} .
\]
The previous inequality follows from the estimates
\begin{align*}
\sum_{i=1}^{\infty} \mu(B_i) \lp \inf_{c_i} m_{|f-c_i|}^s (B_i) \rp^p
\leq \sum_{i=1}^\infty \mu\lp B_i\rp \norm{f}_{\text{BMO}_{0,s}}^p \leq \mu(\Omega) \norm{f}_{\text{BMO}_{0,s}}^p .
\end{align*}

\end{remark}

The median-type John--Nirenberg space $JN_{p,0,s}$ is a generalization of BMO in the sense that a function is in BMO if and only if the $JN_{p,0,s}$ norm of the function is uniformly bounded as $p$ tends to infinity.

\begin{proposition}
If $\Omega \subset X$ has finite measure, then it holds that
\[
\lim_{p \to \infty} \norm{f}_{JN_{p,0,s}(\Omega)} = \norm{f}_{\text{\normalfont BMO}_{0,s}(\Omega)} .
\]

\end{proposition}

\begin{proof}
Let $\{ B_i \}_i$ be a collection of pairwise disjoint balls contained in $\Omega$. 
Remind that if $\mu(A) < \infty$, then $\norm{g}_{L^p(A)} \to \norm{g}_{L^\infty(A)}$ as $p \to \infty$.
By using this, it follows that
\begin{align*}
\lp \sum_{i=1}^{\infty} \mu\lp B_i\rp \lp \inf_{c_i} m_{|f-c_i|}^s (B_i) \rp^p \rp^{\frac{1}{p}}
&= \lp \int_{\Omega} \lp \sum_{i=1}^{\infty} \chi_{ B_i}(x)  \inf_{c_i} m_{|f-c_i|}^s (B_i)  \rp^p \dmu(x) \rp^{\frac{1}{p}} \\
&\to \sup_{x \in \Omega} \sum_{i=1}^{\infty} \chi_{ B_i}(x)  \inf_{c_i} m_{|f-c_i|}^s (B_i)  \\
&= \sup_{i} \inf_{c_i} m_{|f-c_i|}^s (B_i) 
\end{align*}
as $p \to \infty$. Hence, we have
\begin{align*}
\sup_{\{ B_i \}} \lim_{p \to \infty} \lp \sum_{i=1}^{\infty} \mu\lp B_i\rp \lp \inf_{c_i} m_{|f-c_i|}^s (B_i) \rp^p \rp^{\frac{1}{p}}
&= \sup_{\{ B_i \}} \sup_{i} \inf_{c_i} m_{|f-c_i|}^s (B_i)  \\
&= \sup_{B \subset \Omega} \inf_{c} m_{|f-c|}^s (B)  \\
&= \norm{f}_{\text{BMO}_{0,s}} .
\end{align*}
We can interchange the order of taking the supremum and the limit since 
\[
\lp \int_{\Omega} \lp \sum_{j=1}^{\infty} \chi_{ B_j}(x)  \inf_{c_j} m_{|f-c_j|}^s (B_j)  \rp^p \dmu(x) \rp^{\frac{1}{p}}
\]
is an increasing function of $p$ which can be seen by H\"{o}lder's inequality.
Thus, we conclude that
\[
\lim_{p \to \infty} \norm{f}_{JN_{p,0,s}} = \norm{f}_{\text{BMO}_{0,s}} .
\]

\end{proof}

\section{John--Nirenberg lemma for $JN_{p,0,s}$}

We need two lemmas to prove the John--Nirenberg inequality for $JN_{p,0,s}$ which implies that $JN_{p,0,s}$ is contained in $L^{p,\infty}$, that is, weak $L^p$.
The first lemma is a Calder\'{o}n-Zygmund decomposition and the second one is a good-$\lambda$ inequality.

Throughout the argument let $\eta > 0$ and $B_0 = B(x_{B_0},r_{B_0}) \subset X$ be fixed. We denote
$\widehat{B}_0 = (1+\eta) B_0$,
\[
\mathcal{B} = \{B(x_B,r_B) : x_B \in B_0, r_B \leq \eta r_{B_0} \}
\]
and
\begin{equation}
\label{alpha}
\alpha = 5^D c_\mu^2 \lp 1 + \frac{1}{\eta} \rp^D ,
\end{equation}
where $c_\mu$ is the doubling constant and $D = \log_2 c_\mu$ is the doubling dimension.
We define a maximal function
\[
M_{\mathcal{B}}f(x) = \sup_{\substack{B \ni x \\ B \in \mathcal{B}}} m_{|f|}^{t}(B)
\]
with the understanding that $M_\mathcal{B} f(x) = 0$ if there is no ball $B \in \mathcal{B}$ such that $x \in B$. 
In particular, $M_\mathcal{B} f(x) = 0$ for every $x \in X \setminus \widehat{B}_0$. By the Lebesgue differentiation theorem for medians (Lemma~\ref{leb.diff.medians}), we have $|f(x)| \leq M_{\mathcal{B}}f(x)$ for $\mu$-almost every $x \in B_0$.
Moreover, denote 
\[
E_\lambda = E_\lambda^f = \{ x \in \widehat{B}_0 : M_{\mathcal{B}}f(x) > \lambda \} .
\]

\begin{lemma}
\label{lemma2}
If $B \in \mathcal{B}$ such that $m_{f}^{t}(B) >  m_{f}^{t/\alpha}(\widehat{B}_0)$, then it holds that $ r_B \leq \frac{\eta}{5} r_{B_0}$.
\end{lemma}

\begin{proof}
By the assumption, we have
\begin{align*}
t \mu(B) \leq \mu\{ x \in B: f \geq m_{f}^{t}(B) \} \leq \mu\{ x \in \widehat{B}_0: f > m_{f}^{t/ \alpha}(\widehat{B}_0) \} \leq \frac{t}{\alpha} \mu(\widehat{B}_0) .
\end{align*}
Therefore, it holds that
\[
\alpha \leq \frac{\mu(\widehat{B}_0)}{\mu(B)} \leq c_\mu^2 \lp \frac{r_{\widehat{B}_0}}{r_B} \rp^D = c_\mu^2 (1+\eta)^D \lp \frac{r_{B_0}}{r_B} \rp^D ,
\]
which implies $ r_B \leq \frac{\eta}{5} r_{B_0} $ by recalling~\eqref{alpha}.

\end{proof}

The following lemma is a Calder\'{o}n-Zygmund decomposition for medians in metric measure spaces with a doubling measure.

\begin{lemma}
\label{lemma3}
Let $f\geq 0$ be a measurable function defined on $\widehat{B}_0$.
Assume that $E_\lambda \neq \emptyset $ and
\begin{equation*}
m_{f}^{t/\alpha}(\widehat{B}_0) \leq \lambda
\end{equation*}
holds for some $0<t\leq 1$, where $\alpha$ is given in~\eqref{alpha}.
Then there exist countably many pairwise disjoint balls $B_i \in \mathcal{B} $ such that
\begin{enumerate}[(i),topsep=5pt,itemsep=5pt]
\item $ \bigcup_i B_i \subset E_\lambda \subset \bigcup_i 5B_i $,

\item $r_{B_i} \leq \frac{\eta}{5} r_{B_0}$,

\item $ m_{f}^{t}(B_i) > \lambda$,

\item $m_{f}^{t}(\sigma B_i) \leq \lambda$ whenever $ \sigma \geq 2$ and $\sigma B_i \in \mathcal{B}$.

\end{enumerate}
The collection of balls $\{B_i\}_i$ is called the Calder\'{o}n-Zygmund balls $B_{i,\lambda}$ at level $\lambda$. Furthermore, if $m_{f}^{t/\alpha}(\widehat{B}_0) \leq \lambda' \leq \lambda$, then it is possible to choose Calder\'{o}n-Zygmund balls $ B_{j,\lambda'} $ at level $\lambda'$ in a manner that for each $B_{i,\lambda}$ we can find $B_{j, \lambda'}$ such that
$
B_{i,\lambda} \subset 5 B_{j,\lambda'} .
$

\end{lemma}

\begin{proof}
For every $x \in E_\lambda$, denote
\[
r_x(\lambda) = \sup \{ r_B : B \in \mathcal{B}, x \in B, m_{f}^{t}(B) > \lambda \} .
\]
By the assumption, the set over which the supremum is taken is non-empty.
Moreover, Lemma~\ref{lemma2} implies that $r_x(\lambda) \leq \frac{\eta}{5} r_{B_0}$. For every $x \in E_\lambda$, we can find a ball $B_{x, \lambda} \in \mathcal{B}$ with $x \in B_{x, \lambda}$ such that 
 \[
\frac{r_x(\lambda)}{2} < r_{B_{x, \lambda}} \leq r_x(\lambda)
\quad \text{and} \quad  m_{f}^{t}(B_{x, \lambda}) > \lambda .
\] 
We then have $m_{f}^{t}(\sigma B_{x, \lambda}) \leq \lambda$ whenever $ \sigma \geq 2$ and $\sigma B_{x, \lambda} \in \mathcal{B}$.
By applying the 5-covering theorem, we obtain a countable collection of pairwise disjoint balls $\{B_i\}_i$ such that
\[
\bigcup_i B_i \subset E_\lambda \subset \bigcup_i 5B_i.
\]
Hence, so obtained balls $B_i$ are the Calder\'{o}n-Zygmund balls at level $\lambda$ and we denote them by $B_{i,\lambda}$.

We have constructed the Calder\'{o}n-Zygmund decomposition at level $\lambda$ and now focus on $\lambda'$.
Note that $E_{\lambda} \subset E_{\lambda'}$ and $r_x(\lambda) \leq r_x(\lambda')$ for every $x \in E_{\lambda}$.
Thus, for every $x \in E_{\lambda}$, we can choose a ball $B_{x, \lambda'} \in \mathcal{B}$ with $x \in B_{x, \lambda'}$ such that $B_{x, \lambda} \subset B_{x, \lambda'}$,
\[
\frac{r_x(\lambda')}{2} < r_{B_{x, \lambda'}} \leq r_x(\lambda')  \quad \text{and} \quad m_{f}^{t}(B_{x, \lambda'}) > \lambda' .
\]
Whereas, for every $x \in E_{\lambda'} \setminus E_{\lambda}$, we choose the ball $B_{x, \lambda'}$ in the similar way expect we do not have $B_{x, \lambda} \subset B_{x, \lambda'}$.

We then apply the 5-covering theorem to the balls $B_{x, \lambda'}$ to obtain the Calder\'{o}n-Zygmund balls $B_{j,\lambda'}$ at level $\lambda'$.
Moreover, the 5-covering theorem states that for every ball $B_{x, \lambda'}$ there is an enlarged ball $5 B_{j,\lambda'}$ such that 
$
B_{x, \lambda'} \subset 5 B_{j,\lambda'} .
$
Since $B_{x, \lambda} \subset B_{x, \lambda'}$ for every $x \in E_{\lambda}$, it holds that for each $B_{i, \lambda}$ there is $B_{j, \lambda'}$ such that
$
B_{i, \lambda} \subset 5 B_{j,\lambda'} .
$

\end{proof}

We move on to a good-$\lambda$ inequality which is crucial for the proof of the John--Nirenberg inequality.
In the proof of the good-$\lambda$ type inequality for the integral-type $JN_p$ in~\cite{aalto}, all Calder\'{o}n--Zygmund balls can be treated in the same way.
However, in the case of medians, we need to divide Calder\'{o}n--Zygmund balls into two collections which must be considered separately. 
This is due to the fact that medians lack the monotonicity on sets that integrals have.

\begin{lemma}
\label{lemma4}
Let $0< t \leq 1/2$, $K > 1$ and $f \in JN_{p,0,s}(\widehat{B}_0)$ for some $0<s \leq \frac{t}{2K^p c_\mu^3}$. Assume that $E_{K\lambda} \neq \emptyset$ and
\[
m_{|f|}^{t/\alpha}(\widehat{B}_0) \leq \lambda ,
\]
where $\alpha$ is given in~\eqref{alpha}.
Consider collections of Calder\'{o}n-Zygmund balls $\{B_{i,\lambda}\}_i$ and $\{B_{j,K\lambda}\}_j$ for the function $|f|$ such that each $B_{j,K\lambda}$ is contained in some $5B_{i,\lambda}$.
Then it follows that
\[
\sum_{j=1}^\infty \mu(B_{j,K\lambda}) \leq \frac{2^p c_\mu^3}{(K-1)^p}  \frac{\norm{f}_{JN_{p,0,s}(\widehat{B}_0)}^p}{\lambda^p} + \frac{1}{2K^p} \sum_{i=1}^\infty \mu(B_{i,\lambda}) .
\]

\end{lemma}

\begin{proof}
Denote
\[
J_i = \left\{ j \in \mathbb{N}: B_{j,K\lambda} \subset 5B_{i,\lambda}, j \notin \bigcup_{k=1}^{i-1} J_k \right\}
\]
for every $i \in \mathbb{N}$, and
\[
I = \big\{ i\in \mathbb{N}: \mu(B_{i,\lambda}) \leq 2 K^{p} \mu\big( \bigcup_{j \in J_i} B_{j,K \lambda}\big)  \big\} .
\]
Particularly, the set $J_i$ contains those indexes $j \notin \bigcup_{k=1}^{i-1} J_k$ for which $B_{j,K\lambda}$ is contained in $5B_{i,\lambda}$.
Since every $B_{j,K\lambda}$ is contained in some $5B_{i,\lambda}$, we get the partition
\[
\bigcup_{j=1}^\infty B_{j,K\lambda} = \bigcup_{i=1}^\infty \bigcup_{j \in J_i} B_{j,K\lambda} .
\]
Using properties (ii), (v), (vii) of Lemma~\ref{medianprops} and (iii), (iv) of Lemma~\ref{lemma3} in this order, we obtain
\begin{align*}
m_{|f-m_{f}^{t}(5B_{i,\lambda})| }^{t}(B_{j,K\lambda}) &\geq m_{|f| - |m_{f}^{t}(5B_{i,\lambda})| }^{t}(B_{j,K\lambda}) \\
&= m_{|f| }^{t}(B_{j,K\lambda}) - |m_{f}^{t}(5B_{i,\lambda})| \\
&\geq m_{|f| }^{t}(B_{j,K\lambda}) - m_{|f|}^{t}(5B_{i,\lambda}) \\
&\geq K\lambda - \lambda = (K-1)\lambda .
\end{align*}
Since $B_{j,K\lambda}$ are pairwise disjoint, property (x) of Lemma~\ref{medianprops} implies that
\[
m_{|f-m_{f}^{t}(5B_{i,\lambda})| }^{t}\big(\bigcup_{j \in J_i} B_{j,K\lambda}\big) \geq (K-1)\lambda
\]
for every $i \in \mathbb{N}$.
For $i \in I$, it holds that
\[
\mu(5B_{i,\lambda}) \leq c_\mu^3 \mu(B_{i,\lambda}) \leq 2 K^p c_\mu^3 \mu\big( \bigcup_{j \in J_i} B_{j,K\lambda}\big) .
\]
Hence, by property (iii) of Lemma~\ref{medianprops}, we have
\[
(K-1)\lambda \leq m_{|f-m_{f}^{t}(5B_{i,\lambda})| }^{t}\big(\bigcup_{j \in J_i} B_{j,K\lambda}\big) \leq m_{|f-m_{f}^{t}(5B_{i,\lambda})| }^{t/\beta}(5B_{i,\lambda})
\]
for every $i \in I$, where $\beta = 2 K^p c_\mu^3$.

Denote
\[
M_\mathcal{B}^\# f(x) = \sup_{\substack{B \ni x \\ B \in \mathcal{B}}} m_{|f-m_{f}^{t}(B)| }^{t/\beta}(B) .
\]
Then for $i \in I$, we have $M_\mathcal{B}^\# f(x) \geq m_{|f-m_{f}^{t}(5B_{i,\lambda})| }^{t/\beta}(5B_{i,\lambda}) > (K-1)\lambda$ for every $x \in 5B_{i,\lambda}$.
Thus, we get
\begin{align*}
\mu\Big(\bigcup_{i\in I} \bigcup_{j \in J_i} B_{j,K\lambda}\Big) \leq \mu\Big(\bigcup_{i\in I} 5B_{i,\lambda}\Big) \leq \mu(\{ x \in \widehat{B}_0 : M_\mathcal{B}^\# f(x) > (K-1)\lambda \}) .
\end{align*}
For every $x \in \{ x \in \widehat{B}_0 : M_\mathcal{B}^\# f(x) > (K-1)\lambda \}$, there exists $B_x \in \mathcal{B}$ such that $x \in B_x$ and $m_{|f-m_{f}^{t}(B_x)| }^{t/\beta}(B_x) > (K-1)\lambda$.
Applying the 5-covering theorem, we get a countable collection of pairwise disjoint balls $B_k$ such that
\[
\{ x \in \widehat{B}_0 : M_\mathcal{B}^\# f(x) > (K-1)\lambda \} \subset \bigcup_{k=1}^\infty 5 B_k .
\]
We then have
\begin{align*}
\mu(\{ x \in \widehat{B}_0 : M_\mathcal{B}^\# f(x) > (&K-1)\lambda \}) \leq \sum_{k=1}^\infty \mu(5 B_k) \leq c_\mu^3 \sum_{k=1}^\infty \mu( B_k) \\
&\leq c_\mu^3 \frac{1}{(K-1)^p \lambda^p} \sum_{k=1}^\infty \mu( B_k) \lp m_{|f-m_{f}^{t}(B_k)| }^{t/\beta}(B_k) \rp^p \\
&\leq \frac{2^p c_\mu^3}{(K-1)^p}  \frac{\norm{f}_{JN_{p,0,s}(\widehat{B}_0)}^p}{\lambda^p} ,
\end{align*}
whenever $s \leq \frac{t}{\beta} = \frac{t}{2K^p c_\mu^3}$.

On the other hand, for $i \notin I$ we have
\[
\sum_{j \in J_i} \mu(B_{j,K\lambda}) \leq \frac{1}{2K^p} \mu(B_{i,\lambda}),
\]
and thus summing over $i \notin I$ we get
\[
\sum_{i\notin I} \sum_{j \in J_i} \mu(B_{j,K\lambda}) \leq \frac{1}{2K^p} \sum_{i\notin I} \mu(B_{i,\lambda}) .
\]
By combining the cases $i\in I$ and $i \notin I$, we conclude that
\begin{align*}
\sum_{j=1}^\infty \mu(B_{j,K\lambda}) &= \sum_{i\in I} \sum_{j \in J_i} \mu(B_{j,K\lambda}) + \sum_{i\notin I} \sum_{j \in J_i} \mu(B_{j,K\lambda}) \\
&\leq \frac{2^p c_\mu^3}{(K-1)^p}  \frac{\norm{f}_{JN_{p,0,s}(\widehat{B}_0)}^p}{\lambda^p} + \frac{1}{2K^p} \sum_{i=1}^\infty \mu(B_{i,\lambda}) .
\end{align*}

\end{proof}

We now state our main result which is the John--Nirenberg inequality for $JN_{p,0,s}$. It implies that $JN_{p,0,s}(\widehat{B})$ is contained in $L^{p,\infty}(B)$ for all balls $B \subset X$.

\begin{theorem}
\label{thm1}
Let $0 < s \leq s_0 = \min\{\tfrac{1}{2\alpha}, \tfrac{1}{8 c_\mu^3}\} $ and $s \leq r \leq \frac{1}{2}$, where $\alpha$ is given in~\eqref{alpha}.
If $f \in JN_{p,0,s}(\widehat{B}_0)$, then for every $\lambda > 0$ it holds that
\[
\mu(\{ x \in B_0: |f(x)-m_f^{r}(B_0)| > \lambda \}) \leq c \frac{\norm{f}_{JN_{p,0,s}(\widehat{B}_0)}^p}{\lambda^p} ,
\]
where the constant $c$ depends on $p$ and the doubling constant $c_\mu$, that is,
\[
c = \frac{2^{p+3} c_\mu^6}{(2^\frac{1}{p} -1)^p} .
\]
\end{theorem}

\begin{proof}
Let $t = \frac{1}{2}$.
We can assume that $E_\lambda = E_\lambda^{|f-m_f^{r}(B_0)|} \neq \emptyset$, since otherwise the claim is clear. 
In addition, we can assume that $\lambda > \lambda_0 = m_{|f-m_f^{r}(B_0)|}^{t/\alpha}(\widehat{B}_0)$, since otherwise if $0 < \lambda \leq \lambda_0$, then by the trivial estimate and Lemma~\ref{lemmamedian1} we have
\begin{align*}
\mu(\{ x \in B_0: |f(x)-m_f^{r}(B_0)| > \lambda \}) &\leq \mu(\widehat{B}_0) \frac{\lp m_{|f-m_f^{r}(B_0)|}^{t/\alpha}(\widehat{B}_0) \rp^p}{\lambda^p} \\
&\leq 2^p \frac{\norm{f}_{JN_{p,0,s}}^p}{\lambda^p} ,
\end{align*}
whenever $s \leq \frac{t}{\alpha} = \frac{1}{2\alpha}$ and $s \leq r \leq \frac{1}{2}$.
Thus, the conditions in Lemma~\ref{lemma3} and Lemma~\ref{lemma4} hold for the function $|f-m_f^{r}(B_0)|$.

Let $K = 2^{1/p}$ and $N \in \mathbb{N}$ such that
\[
K^N \lambda_0 < \lambda \leq K^{N+1} \lambda_0 .
\]
Consider $N+1$ families of Calder\'{o}n-Zygmund balls at levels $\lambda_0, K \lambda_0, \dots, K^N \lambda_0$.
Note that for every $n=0,1,\dots,N-1,$ each $B_{i,K^{n+1}\lambda}$ is contained in some $5B_{j,K^n\lambda}$.
It follows that
\begin{align*}
\mu(\{ x \in B_0: |f(x)-m_f^{r}(B_0)| > \lambda \}) &\leq \mu(\{ x \in B_0: |f(x)-m_f^{r}(B_0)| > K^N \lambda_0 \}) \\ &\leq \sum_{j=1}^\infty \mu(5B_{j,K^N \lambda_0}) \leq c_\mu^3 \sum_{j=1}^\infty \mu(B_{j,K^N \lambda_0}) .
\end{align*}
We claim that
\[
\sum_{i=1}^\infty \mu(B_{i,K^n \lambda_0}) \leq c_1  \frac{\norm{f}_{JN_{p,0,s}}^p}{(K^n \lambda_0)^p}
\]
for every $n=0,1,\dots,N$, where 
\[
c_1 = \frac{2^{p+1} c_\mu^3 K^p}{(K-1)^p} .
\]
We prove the claim by induction. First, note that the claim holds for $n=0$  since
\[
\sum_{i=1}^\infty \mu(B_{i, \lambda_0}) \leq \mu(\widehat{B}_0) = \mu(\widehat{B}_0) \frac{ \lp m_{|f-m_f^{r}(B_0)|}^{t/\alpha}(\widehat{B}_0) \rp^p}{\lambda_0^p} \leq 2^p \frac{ \norm{f}_{JN_{p,0,s}}^p}{\lambda_0^p}.
\]
Assume then that the claim holds for $k \in \{0,1,\dots,N-1\}$, that is, 
\[
\sum_{i=1}^\infty \mu(B_{i,K^k\lambda_0}) \leq c_1 \frac{\norm{f}_{JN_{p,0,s}}^p}{(K^{k} \lambda_0)^p} .
\]
We show that this implies the claim for $k+1$.
By using Lemma~\ref{lemma4} for $K^k \lambda_0$, we observe that
\begin{align*}
\sum_{j=1}^\infty \mu(B_{j,K^{k+1}\lambda_0}) &\leq \frac{2^p c_\mu^3}{(K-1)^p}  \frac{\norm{f}_{JN_{p,0,s}}^p}{(K^k\lambda_0)^p} + \frac{1}{2K^p} \sum_{i=1}^\infty \mu(B_{i,K^k \lambda_0}) \\
&\leq \frac{2^p c_\mu^3}{(K-1)^p} \frac{\norm{f}_{JN_{p,0,s}}^p}{(K^k\lambda_0)^p} + \frac{c_1}{2K^p} \frac{\norm{f}_{JN_{p,0,s}}^p}{(K^{k} \lambda_0)^p} \\
&= \bigg( \frac{2^p c_\mu^3 K^p}{(K-1)^p} + \frac{c_1}{2} \bigg) \frac{\norm{f}_{JN_{p,0,s}}^p}{(K^{k+1} \lambda_0)^p} \\
&= c_1 \frac{\norm{f}_{JN_{p,0,s}}^p}{(K^{k+1} \lambda_0)^p} .
\end{align*}
Therefore, the claim holds for $k+1$.

Hence, we conclude that
\begin{align*}
\mu(\{ x \in B_0: |f(x)-m_f^{r}(B_0)| > \lambda \}) &\leq c_\mu^3 c_1  \frac{\norm{f}_{JN_{p,0,s}}^p}{(K^N \lambda_0)^p} \\
&= c_\mu^3 c_1  K^p \frac{\norm{f}_{JN_{p,0,s}}^p}{(K^{N+1}\lambda_0)^p} \\ 
&\leq c \frac{\norm{f}_{JN_{p,0,s}}^p}{\lambda^p} ,
\end{align*}
where
\[
c = c_\mu^3 c_1  K^p = \frac{2^{p+1} c_\mu^6 K^{2p}}{(K-1)^p} = \frac{2^{p+3} c_\mu^6}{(2^\frac{1}{p} -1)^p} ,
\]
since $K = 2^{1/p}$. \qedhere

\end{proof}

\section{Global John--Nirenberg inequality for $JN_{p,0,s}$ in Boman sets}

We give a proof for the global John--Nirenberg inequality for $JN_{p,0,s}$ in Boman sets.
For more detailed discussion about Boman sets, see~\cite{marolasaari} and references therein.

\begin{definition}
\label{boman_def}
A set $\Delta \subset X$ is called Boman if there are constants $C_2 > C_1 > 1$, $C_3 > 1$, $\rho >1$ and $M \in \mathbb{N}$ and a collection of pairwise disjoint balls $\mathcal{F}$ such that
\begin{enumerate}[(i),topsep=5pt,itemsep=5pt]

\item $\Delta = \bigcup_{B \in \mathcal{F}} C_1 B = \bigcup_{B \in \mathcal{F}} C_2 B$.

\item If $B \in \mathcal{F}$, there are at most $M$ balls $V \in \mathcal{F}$ with $C_2 V \cap C_2 B \neq \emptyset$.

\item There is a central ball $B_\ast \in \mathcal{F}$ such that for each $B \in \mathcal{F}$ there exists a finite collection of balls $\mathcal{C}(B) = \{B_i\}_{i=1}^{k_B} \subset \mathcal{F}$ with $B_1 = B_\ast$ and $B_{k_B} = B$.

\item In $\mathcal{C}(B)$, for each pair of balls $B_i$ and $B_{i-1}$ corresponding to consecutive indices there exists a ball $D_i \subset C_1 B_i \cap C_1 B_{i-1}$ such that $\mu(D_i) \geq C_3 (\mu(B_i) + \mu(B_{i-1}) )$.

\item If $V \in \mathcal{C}(B)$, then $B \subset \rho V$.

\end{enumerate}

\end{definition}

Parameters $C_1,C_2,C_3,\rho$ and $M$ in the results below are the same as in Definition~\ref{boman_def}.
The proof of the following lemma can be found in~\cite{marolasaari} for integral averages. The proof is identical for medians, and thus is omitted here.

\begin{lemma}
\label{bomanlemma}
Let $\Delta \subset X$ be a Boman set, $1 < p < \infty$ and $0 < s \leq 1/2$. Then
\[
\sum_{B \in \mathcal{F}} \norm{ m_f^{s}(C_1 B) - m_f^{s}(C_1 B_\ast) }_{L^{p,\infty}(C_1 B)}^p \leq C_0 \sum_{V \in \mathcal{F}} \norm{ f - m_f^{s}(C_1 V) }_{L^{p,\infty}(C_1 V)}^p ,
\]
where the constant $C_0$ depends on $p$, the doubling constant $c_\mu$, $C_1,C_2,C_3,\rho$ and $M$.

\end{lemma}

For the next theorem, the global John--Nirenberg lemma, we fix the parameter~$\eta$ from Section 4 such that $1 + \eta = \frac{C_2}{C_1}$. The proof follows that of~\cite{marolasaari}.

\begin{theorem}
\label{JN-lemma_Boman}
Let $\Delta \subset X$ be a Boman set and $0 < s \leq s_0$, where $s_0$ is given in Theorem~\ref{thm1}.
If $f \in JN_{p,0,s}(\Delta)$, then there exists $a \in \mathbb{R}$ such that for every $\lambda > 0$ it holds that
\[
\mu(\{ x \in \Delta: |f(x)-a| > \lambda \}) \leq C \frac{\norm{f}_{JN_{p,0,s}(\Delta)}^p}{\lambda^p} ,
\]
where $C$ depends on $p$, the doubling constant $c_\mu$, $C_1,C_2,C_3,\rho$ and $M$.

\end{theorem}

\begin{proof}
Let $s \leq r \leq \frac{1}{2}$, $\mathcal{F}$ be the collection of balls in the definition of the Boman set $\Delta$ and $B_\ast$ be the central ball. 
We have
\begin{align*}
\mu(\{ x \in \Delta: |f - m_f^{r}&(C_1 B_\ast)|  > \lambda \}) \leq \sum_{B \in \mathcal{F}} \mu(\{ x \in C_1 B: |f-m_f^{r}(C_1 B_\ast)| > \lambda \}) \\
&\leq \sum_{B \in \mathcal{F}} \mu(\{ x \in C_1 B: |f-m_f^{r}(C_1 B)| > \lambda/2 \}) \\
&\qquad + \sum_{B \in \mathcal{F}} \mu(\{ x \in C_1 B: |m_f^{r}(C_1 B) - m_f^{r}(C_1 B_\ast)| > \lambda/2 \}) \\
&= I_1 + I_2 .
\end{align*}
By applying Theorem \ref{thm1}, we get
\begin{align*}
I_1 \leq 2^p c \sum_{B \in \mathcal{F}} \frac{\norm{f}_{JN_{p,0,s}(C_2 B)}^p}{\lambda^p} .
\end{align*}
To estimate the second term, we use the definition of weak $L^p$ norm, Lemma~\ref{bomanlemma} and Theorem~\ref{thm1} to obtain
\begin{align*}
\lp \frac{\lambda}{2} \rp^p I_2 &\leq \sum_{B \in \mathcal{F}} \norm{ m_f^{r}(C_1 B) - m_f^{r}(C_1 B_\ast) }_{L^{p,\infty}(C_1 B)}^p \\
&\leq C_0 \sum_{B \in \mathcal{F}} \norm{ f - m_f^{r}(C_1 B) }_{L^{p,\infty}(C_1 B)}^p \\
& = C_0 \sum_{B \in \mathcal{F}} \sup_{\gamma >0} \gamma^p \mu(\{ x \in C_1 B: |f-m_f^{r}(C_1 B)| > \gamma \}) \\
&\leq C_0 \sum_{B \in \mathcal{F}} c \norm{f}_{JN_{p,0,s}(C_2 B)}^p .
\end{align*}
Therefore, it holds that
\[
I_1 + I_2 \leq \frac{\widetilde{C}}{\lambda^p} \sum_{B \in \mathcal{F}} \norm{f}_{JN_{p,0,s}(C_2 B)}^p ,
\]
where $\widetilde{C} = 2^p c (C_0 + 1)$.
By property (ii) of Lemma~\ref{boman_def}, the collection $\{C_2 B\}_{B \in \mathcal{F}}$ consists of balls that intersect at most $M$ balls of the same collection.
Thus, it can be decomposed into at most $M$ collections of pairwise disjoint balls $\mathcal{D}_i, i=1,\dots, M$ such that $\{C_2 B\}_{B \in \mathcal{F}} = \bigcup_{i=1}^M \mathcal{D}_i$.
This implies that
\begin{align*}
\mu(\{ x \in \Delta: |f - m_f^{r}(C_1 B_\ast)|  > \lambda \}) &\leq \frac{\widetilde{C}}{\lambda^p} \sum_{B \in \mathcal{F}} \norm{f}_{JN_{p,0,s}(C_2 B)}^p \\
&= \frac{\widetilde{C}}{\lambda^p} \sum_{i=1}^M \sum_{B \in \mathcal{D}_i} \norm{f}_{JN_{p,0,s}(C_2 B)}^p \\
&\leq \frac{C}{\lambda^p} \norm{f}_{JN_{p,0,s}(\Delta)}^p ,
\end{align*}
where $C = \widetilde{C} M$.
This concludes the proof.

\end{proof}

If all balls are Boman sets with uniform parameters, then the median-type John--Nirenberg space coincides with the integral-type John--Nirenberg space in every open set.
For example, geodesic spaces satisfy the uniform Boman condition on balls~\cite{koskela}.

\begin{corollary}
\label{equivalence_Boman}
Let $1<p<\infty$, $0 < q < p$ and $0 < s \leq s_0$, where $s_0$ is given in Theorem~\ref{thm1}.
Assume that all balls in $X$ are Boman sets with uniform parameters $C_1,C_2,C_3,\rho$ and $M$.
Then for every open set $\Omega \subset X$ it holds that
\[
s^\frac{1}{q} \norm{f}_{JN_{p,0,s}(\Omega)} \leq \norm{f}_{JN_{p,q}(\Omega)} \leq \lp \frac{Cp}{p-q} \rp^\frac{1}{q} \norm{f}_{JN_{p,0,s}(\Omega)},
\]
where $C$ is the constant from Theorem~\ref{JN-lemma_Boman}.

\end{corollary}

\begin{proof}

Let $\{B_i\}_i$ be a countable collection of pairwise disjoint balls contained in $\Omega$.
The first inequality is stated and proven in Proposition~\ref{LpinJNp}.
For the second inequality, since by the assumption the balls $B_i$ are Boman sets, we may apply Theorem~\ref{JN-lemma_Boman} on $B_i$. This together with Cavalieri's principle implies
\begin{align*}
\int_{B_i} |f-a|^q \dmu &= q \int_0^\infty \lambda^{q-1} \mu(\{ x \in B_i: |f-a| > \lambda \}) \dla \\
&\leq q \int_{\mu(B_i)^{-\frac{1}{p}} \norm{f}_{JN_{p,0,s}(B_i)}}^\infty C \lambda^{q-p-1} \norm{f}_{JN_{p,0,s}(B_i)}^p \dla \\ &\qquad + q \int_0^{\mu(B_i)^{-\frac{1}{p}} \norm{f}_{JN_{p,0,s}(B_i)}} \lambda^{q-1} \mu(B_i) \dla \\
&= \frac{Cq}{p-q} \mu(B_i)^{1-\frac{q}{p}} \norm{f}_{JN_{p,0,s}(B_i)}^q + \mu(B_i)^{1-\frac{q}{p}} \norm{f}_{JN_{p,0,s}(B_i)}^q \\
&\leq \frac{Cp}{p-q} \mu(B_i)^{1-\frac{q}{p}} \norm{f}_{JN_{p,0,s}(B_i)}^q ,
\end{align*}
where $C$ is the constant from Theorem~\ref{JN-lemma_Boman}.
We then estimate
\begin{align*}
\sum_{i=1}^\infty \mu(B_i) \lp \inf_{c_i} \dashint_{B_i} |f-c_i|^q \dmu \rp^\frac{p}{q} &\leq \sum_{i=1}^\infty \mu(B_i) \lp \dashint_{B_i} |f-a|^q \dmu \rp^\frac{p}{q} \\
&\leq \lp \frac{Cp}{p-q} \rp^\frac{p}{q} \sum_{i=1}^\infty \norm{f}_{JN_{p,0,s}(B_i)}^p \\
&\leq \lp \frac{Cp}{p-q} \rp^\frac{p}{q} \norm{f}_{JN_{p,0,s}(\Omega)}^p .
\end{align*}
Thus, we conclude that
\[
\norm{f}_{JN_{p,q}(\Omega)} \leq \lp \frac{Cp}{p-q} \rp^\frac{1}{q} \norm{f}_{JN_{p,0,s}(\Omega)}.
\]

\end{proof}

\end{document}